\documentclass[12pt]{article}
\usepackage{graphicx}
\usepackage{amsmath}

\newtheorem{theorem}{Theorem}

\newtheorem{corollary}[theorem]{Corollary}

\newtheorem{example}[theorem]{Example}

\newtheorem{remark}[theorem]{Remark}

\newenvironment{proof}[1][Proof]{\textbf{#1.} }{\ \rule{0.5em}{0.5em}}
\input{tcilatex}
\begin{document}

\author{Marek Galewski \\
Faculty of Mathematics and Computer Science, \\
University of Lodz,\\
Banacha 22, 90-238 Lodz, Poland, \\
galewski@math.uni.lodz.pl}
\title{Dependence on parameters for discrete second order boundary value
problems}
\maketitle

\begin{abstract}
We investigate the dependence on parameters for second order difference
equations with two point boundary value conditions by using a variational
method in case when the corresponding Euler action functional is coercive.
Some applications for discrete Emden-Fowler equation are also given.
\end{abstract}



\textbf{Note:} The final version of this submission will be published in
Journal of Difference Equations and Applications

\section{Introduction}

The variational approach towards the existence of solutions to nonlinear
difference equations received some serious attention, see for example, \cite%
{gueFIRST}, \cite{agrawal}, \cite{caiYu}, \cite{Li}, \cite{sehlik}, \cite%
{YanZhangAPMacc}. Various types of methods have so far been employed, in
fact the approaches valid for boundary value problems for differential
equations have successfully been adapted and somehow extended due to the
fact that in the setting of difference equations the boundary value problems
are considered in a finite dimensional space.

In this work we mainly intend to investigate the dependence on a functional
parameter $u$ for the coercive second order boundary value problems taking
as an example the problem originally considered in \cite{Li} by variational
method and in \cite{rach} by the lower-upper function method. Later such
boundary value problem has been reconsidered in a variational formulation in 
\cite{galewskiCOERCIVE} with weaker assumptions than those of \cite{Li}.
However, both the approach of \cite{Li} and the topological method from \cite%
{rach} yield the same existence result - with the same assumptions as is
shown in \cite{galewskiCOERCIVE} - it is the variational method which, in
our opinion, allows for considering the dependence of the solution on a
parameter in some systematic way.

The approach towards investigation of a dependence on a functional parameter
for solutions of ODE in case of coercive action functional originates for
example from \cite{LedzewiczWalczak}. We also base on some of ideas from 
\cite{LedzewiczWalczak} but we put them in a different context and for a
discrete problem. Such an investigation has not been undertaken yet to the
best of our knowledge. The idea of the continuous dependence on parameters
could be summarized as follows: we consider a discrete boundary value
problem which is subject to certain (functional) parameter and which has a
solution with respect to any parameter (function). Therefore corresponding
to a sequence of parameters there exists a sequence of solutions. Supposing
that the sequence of parameters is convergent (in a suitable sense) we
arrive at the limit of a sequence of solutions, which itself is a solution
to the considered problem corresponding to the limit of the parameter
sequence. What is important and what constitutes the main point is that all
solutions, both in the sequence and the limit one, share the same properties.

\section{Dependence on parameters for second order coercive problem}

Before we provide the statement of the problem under consideration, we
introduce some notation. In what follows by $\left[ A,B\right] $ we mean the
discrete interval $\left\{ A,...,B\right\} $. $C\left( \left[ A,B\right]
,R\right) $ is a space of functions $u:\left[ A,B\right] \rightarrow R$
(defined on a discrete interval, and thus necessarily continuous) equipped
with classical maximum norm $\left\| u\right\| _{C}=\max_{k\in \left\{
A,...,B\right\} }\left| u\left( k\right) \right| $.

Let $M>0$ be fixed. A parameter function $u$ belongs to 
\begin{equation*}
L_{M}=\left\{ u\in C\left( \left[ 1,T\right] ,R\right) :\left\Vert
u\right\Vert _{C}\leq M\right\} .
\end{equation*}

$\Delta $ denotes the forward difference operator, i.e. $\Delta x\left(
k\right) =x\left( k+1\right) -x\left( k\right) $. $E$ stands for the space
of functions $y:\left[ 0,T+1\right] \rightarrow R$ such that $y\left(
0\right) =y\left( T+1\right) =0$ considered with norm $\left\Vert
y\right\Vert =\sqrt{\sum_{k=1}^{T}\left( \Delta y\left( k\right) \right) ^{2}%
}.$ By $\left\vert \cdot \right\vert $ we denote Euclidean norm on $E$ and
we see that 
\begin{equation}
\gamma \left\vert y\right\vert \leq \left\Vert y\right\Vert \leq \gamma
_{1}\left\vert y\right\vert \text{ for all\thinspace }y\in E  \label{poin}
\end{equation}
for certain constants $\gamma ,\gamma _{1}>0$ which do not depend on $y$.
\bigskip

In this section we will investigate the following problem in $E$ 
\begin{equation}
\Delta \left( p\left( k\right) \Delta x\left( k-1\right) \right) +f\left(
k,x\left( k\right) ,u\left( k\right) \right) =g\left( k\right)  \label{row}
\end{equation}
which is subject to a parameter $u\in L_{M}$ and which satisfies the
Dirichlet boundary conditions 
\begin{equation}
x\left( 0\right) =x\left( T+1\right) =0.  \label{bd}
\end{equation}
We will assume that$\bigskip $

\textbf{A1} $f\in C\left( \left[ 1,T\right] \times R\times \left[ -M,M\right]
,R\right) ,$ $p\in C\left( \left[ 1,T+1\right] ,R\right) ,$ $g\in C\left( %
\left[ 1,T\right] ,R\right) ;\bigskip $

\textbf{A2} there exists $\alpha >0$ such that $yf\left( k,y,u\right) \leq 0$
for all $\left\vert y\right\vert \geq \alpha $, $\left\vert u\right\vert
\leq M$ and $k=1,...,T;\bigskip $

\textbf{A3} $m=\min_{k\in \left\{ 1,...,T+1\right\} }p\left( k\right)
>0.\bigskip $

Here $f\in C\left( \left[ 1,T\right] \times R\times \left[ -M,M\right]
,R\right) $ means that for each $k\in \left\{ 1,2,...,T\right\} $ function $%
f\left( k,\cdot ,\cdot \right) $ is continuous on $R\times \left[ -M,M\right]
$. Let for $y\in E$ 
\begin{equation*}
F\left( k,y\left( k\right) ,u\left( k\right) \right) =\int_{0}^{y\left(
k\right) }f\left( k,t,u\left( k\right) \right) dt.
\end{equation*}
With assumptions \textbf{A1}-\textbf{A3} the action functional $%
J:E\rightarrow R$ corresponding to (\ref{row})-(\ref{bd}) with a fixed
function $u\in L_{M}$ reads 
\begin{equation}
J_{u}\left( y\right) =\sum_{k=1}^{T+1}\left[ \frac{p\left( k\right) }{2}%
\Delta y^{2}\left( k-1\right) \right] -\sum_{k=1}^{T}F\left( k,y\left(
k\right) ,u\left( k\right) \right) +\sum_{k=1}^{T}g\left( k\right) y\left(
k\right)  \label{gwiazdka}
\end{equation}
and it is coercive and continuous on $E$. Since it is obviously
differentiable in the sense of G\^{a}teaux with bounded G\^{a}teaux
variation at each point, it admits at least one minimizer satisfying (\ref%
{row})-(\ref{bd}), see \cite{galewskiCOERCIVE}, \cite{Li} for details.
Namely, for any fixed $u\in L_{M}$ the set which consists of the arguments
of a minimum to $J_{u}$ 
\begin{equation*}
V_{u}=\left\{ x\in E:J_{u}\left( x\right) =\inf_{v\in E}J_{u}\left( v\right) 
\text{ and }\frac{d}{dx}J_{u}\left( x\right) =0\right\}
\end{equation*}
is non-empty. We will investigate the behavior of the sequence $\left\{
x_{n}\right\} _{n=1}^{\infty }$ of solutions to (\ref{row})-(\ref{bd})
depending on the convergence of the sequence of parameters $\left\{
u_{n}\right\} _{n=1}^{\infty }$. Moreover, we consider the case of the
existence and dependence on parameters for positive solutions. Next, we
investigate some general stability results in a sense which we describe
later. In fact the dependence on a parameter is obtained as a special case
of stability which we show by giving the alternative proof of the main
result, namely Theorem \ref{dep_param_theo}.

We would like to mention that typically with (\ref{row})-(\ref{bd}) it is
associated the following functional instead of (\ref{gwiazdka}) 
\begin{equation*}
J_{u}^{1}\left( y\right) =\sum_{k=1}^{T+1}\left[ \frac{p\left( k\right) }{2}%
\Delta y^{2}\left( k-1\right) -F\left( k,y\left( k\right) ,u\left( k\right)
\right) +g\left( k\right) y\left( k\right) \right] .
\end{equation*}
However, it requires that $g,$ $f\in C\left( \left[ 1,T+1\right] ,R\right) $%
. As in \cite{Li} we can show that (\ref{row})-(\ref{bd}) stands for
critical point to (\ref{gwiazdka}) as well.

\subsection{Dependence on parameters}

\begin{theorem}
\label{dep_param_theo}Assume \textbf{A1}-\textbf{A3}. For any fixed $u\in
L_{M}$ there exists at least one solution $x\in V_{u}$ to problem (\ref{row}%
)-(\ref{bd}). Let $\left\{ u_{n}\right\} _{n=1}^{\infty }\subset L_{M}$ be a
convergent sequence of parameters, where $\lim_{n\rightarrow \infty }u_{n}=%
\overline{u}\in L_{M}$. For any sequence $\left\{ x_{n}\right\}
_{n=1}^{\infty }$ of solutions $x_{n}\in V_{u_{n}}$ to the problem (\ref{row}%
)-(\ref{bd}) corresponding to $u_{n}$, there exist a subsequence $\left\{
x_{n_{i}}\right\} _{i=1}^{\infty }\subset E$ and an element $\overline{x}\in
E$ such that $\lim_{i\rightarrow \infty }x_{n_{i}}=\overline{x}$ and $J_{%
\overline{u}}\left( \overline{x}\right) =\inf_{y\in E}J_{\overline{u}}\left(
y\right) $. Moreover, $\overline{x}\in V_{\overline{u}},$ i.e. $\overline{x}$
satisfies 
\begin{equation*}
\Delta \left( p\left( k\right) \Delta \overline{x}\left( k-1\right) \right)
+f\left( k,\overline{x}\left( k\right) ,\overline{u}\left( k\right) \right)
=g\left( k\right) \text{, }\overline{x}\left( 0\right) =\overline{x}\left(
T+1\right) =0.
\end{equation*}
\end{theorem}

\begin{proof}
From \cite{galewskiCOERCIVE} it follows that for each $n=1,2,...$ there
exists a solution $x_{n}\in V_{u_{n}}$ to (\ref{row})-(\ref{bd}) . We see
that the sequence $\left\{ x_{n}\right\} _{n=1}^{\infty }$ is bounded.
Indeed, for any $n$ we have $x_{n}\in V_{u_{n}}\subset \left\{
x:J_{u_{n}}\left( x\right) \leq J_{u_{n}}\left( 0\right) \right\} $. By 
\textbf{A2} we further obtain for some $C>0$ and for $\ $all $x_{n}\in
V_{u_{n}}$ 
\begin{equation}
\begin{array}{l}
\sum_{k=1}^{T}F\left( k,x_{n}\left( k\right) ,u_{n}\left( k\right) \right)
=\sum_{k=1}^{T}\int_{0}^{x_{n}\left( k\right) }f\left( k,t,u_{n}\left(
k\right) \right) dt\leq \bigskip \\ 
\sum_{k=1}^{T}\int_{-\alpha }^{\alpha }\left| f\left( k,x_{n}\left( k\right)
,u_{n}\left( k\right) \right) \right| \leq C.%
\end{array}
\label{c}
\end{equation}
Next, by (\ref{c}) and by (\ref{poin}) we get 
\begin{equation}
\begin{array}{l}
J_{u_{n}}\left( x_{n}\right) =\sum_{k=1}^{T+1}\left[ \frac{p\left( k\right) 
}{2}\Delta x_{n}^{2}\left( k-1\right) \right] -\sum_{k=1}^{T}F\left(
k,x_{n}\left( k\right) ,u_{n}\left( k\right) \right) \bigskip \\ 
+\sum_{k=1}^{T}g\left( k\right) x_{n}\left( k\right) \geq \frac{m}{2}\left\|
x_{n}\right\| ^{2}-\sqrt{\sum_{k=1}^{T}g^{2}\left( k\right) }\left|
x_{n}\right| -C\geq \bigskip \\ 
\frac{m}{2}\left\| x_{n}\right\| ^{2}-\frac{1}{\gamma }\sqrt{%
\sum_{k=1}^{T}g^{2}\left( k\right) }\left\| x_{n}\right\| -C.%
\end{array}
\label{coercive}
\end{equation}
On the other hand we see by definition of $F$ that $-F\left( k,0,u_{n}\left(
k\right) \right) =0$, so 
\begin{equation*}
J_{u_{n}}\left( x_{n}\right) \leq J_{u_{n}}\left( 0\right) =0.
\end{equation*}
Thus 
\begin{equation}
\frac{m}{2}\left\| x_{n}\right\| ^{2}-\frac{1}{\gamma }\sqrt{%
\sum_{k=1}^{T}g^{2}\left( k\right) }\left\| x_{n}\right\| \leq C\text{. }
\label{nier_kwadrat}
\end{equation}
Since (\ref{nier_kwadrat}) treated as a quadratic inequality with variable $%
t=\left\| x_{n}\right\| $ has solutions in a bounded closed interval and
since $n$ was fixed arbitrarily, we see that $\left\{ x_{n}\right\}
_{n=1}^{\infty }$ is bounded in $E$. Hence, it has a convergent subsequence $%
\left\{ x_{n_{i}}\right\} _{i=1}^{\infty }$. We denote its limit by $%
\overline{x}$. (We note that in \cite{galewskiCOERCIVE} relation (\ref%
{coercive}) is used in demonstrating that the action functional is indeed
coercive. )$\bigskip $

Now we demonstrate that $\overline{x}$ satisfies (\ref{row})-(\ref{bd})
corresponding to $\overline{u}$. Firstly, we observe that there exists $%
x_{0}\in E$ such that $x_{0}$ solves (\ref{row})-(\ref{bd}) with $\overline{u%
}$ and $J_{\overline{u}}\left( x_{0}\right) =\inf_{y\in E}J_{\overline{u}%
}\left( y\right) $. We see that there are two possibilities: namely either $%
J_{\overline{u}}\left( x_{0}\right) <J_{\overline{u}}\left( \overline{x}%
\right) $ or $J_{\overline{u}}\left( x_{0}\right) =J_{\overline{u}}\left( 
\overline{x}\right) $. On the one hand we suppose that $J_{\overline{u}%
}\left( x_{0}\right) <J_{\overline{u}}\left( \overline{x}\right) $. Now
there exists a constant $\delta >0$ such that in fact 
\begin{equation}
J_{\overline{u}}\left( \overline{x}\right) -J_{\overline{u}}\left(
x_{0}\right) >\delta >0.  \label{beta}
\end{equation}
We investigate the convergence of the right hand side of the inequality 
\begin{equation}
\delta <\left( J_{u_{n_{i}}}\left( x_{n_{i}}\right) -J_{\overline{u}}\left(
x_{0}\right) \right) -\left( J_{u_{n_{i}}}\left( x_{n_{i}}\right)
-J_{u_{n_{i}}}\left( \overline{x}\right) \right) -\left( J_{u_{n_{i}}}\left( 
\overline{x}\right) -J_{\overline{u}}\left( \overline{x}\right) \right)
\label{granice}
\end{equation}
which is equivalent to (\ref{beta}). It is obvious, by continuity, that 
\begin{equation}
\lim_{i\rightarrow \infty }\left( J_{u_{n_{i}}}\left( \overline{x}\right)
-J_{\overline{u}}\left( \overline{x}\right) \right) =0\text{ and }%
\lim\nolimits_{i\rightarrow \infty }\left( J_{u_{n_{i}}}\left(
x_{n_{i}}\right) -J_{u_{n_{i}}}\left( \overline{x}\right) \right) =0.
\label{gr_pom}
\end{equation}
We also have 
\begin{equation}
\lim_{i\rightarrow \infty }\left( J_{u_{n_{i}}}\left( x_{0}\right) -J_{%
\overline{u}}\left( x_{0}\right) \right) =0.  \label{gw}
\end{equation}
Since $x_{n_{i}}$ minimizes $J_{u_{n_{i}}}$ over $E$ we see that $%
J_{u_{n_{i}}}\left( x_{n_{i}}\right) \leq J_{u_{n_{i}}}\left( x_{0}\right) $
for any $n_{i}$. Therefore, we get by (\ref{gw}) 
\begin{equation*}
\lim_{i\rightarrow \infty }\left( J_{u_{n_{i}}}\left( x_{n_{i}}\right) -J_{%
\overline{u}}\left( x_{0}\right) \right) \leq \lim_{i\rightarrow \infty
}\left( J_{u_{n_{i}}}\left( x_{0}\right) -J_{\overline{u}}\left(
x_{0}\right) \right) =0.
\end{equation*}
Now we obtain $\delta \leq 0$ in (\ref{granice}), which is a contradiction.
Thus $J_{\overline{u}}\left( \overline{x}\right) =\inf_{y\in E}J_{\overline{u%
}}\left( y\right) $ and since $J_{\overline{u}}$ is differentiable in the
sense of G\^{a}teaux we have $\overline{x}\in V_{\overline{u}}$. Hence $%
\overline{x}$ necessarily satisfies (\ref{row})-(\ref{bd}). On the other
hand, if we have $J_{\overline{u}}\left( x_{0}\right) =J_{\overline{u}%
}\left( \overline{x}\right) $ the result readily follows.
\end{proof}

\subsection{Case of positive solutions}

It remains to consider the question of the existence and the dependence on
parameters for positive solutions. The approach of \cite{rach} allows for
obtaining at least one positive solution to (\ref{row})-(\ref{bd}) with some
assumptions added to those leading to the existence result. In fact, the
same holds true for the variational formulation although with modified
assumptions. We must add some assumption to \textbf{A1}, \textbf{A2}, 
\textbf{A3} and modify \textbf{A4 }in assumptions \textbf{A1}, \textbf{A3, A4%
}. Namely, we assume that $\bigskip $

\textbf{A5 }$f\left( k,y,u\right) -g\left( k\right) \geq 0$ for all $k\in %
\left[ 1,T\right] $, all $y\in R$ and all $\left\vert u\right\vert \leq M;$
there exists $k_{1}\in \left[ 1,T\right] $ such that $f\left(
k_{1},y,u\right) -g\left( k_{1}\right) >0$ for all $y\in R$ and all $%
\left\vert u\right\vert \leq M;\bigskip $

\textbf{A6 }$\lim_{y\rightarrow \infty }\sum_{k=1}^{T}F\left( k,y,u\right)
=-\infty $ and $\lim_{y\rightarrow -\infty }\sum_{k=1}^{T}F\left(
k,y,u\right) =c\in R$ uniformly in $\left\vert u\right\vert \leq M.\bigskip $

\begin{remark}
Assumption \textbf{A6}\ is different from \textbf{A4}. Indeed, function $%
F\left( x\right) =-e^{x}$ satisfies \textbf{A6}\ and it does not satisfy 
\textbf{A4}\ while function $F\left( x\right) =-x^{l}$ for any even $l$
satisfies \textbf{A4}\ and it does not satisfy \textbf{A6}. Still both
assumptions \textbf{A4 }and \textbf{A6 }yield that functional $J_{u}$ is
coercive.
\end{remark}

We recall that by a positive solution to (\ref{row})-(\ref{bd}) we mean such
a function $x\in E$ which satisfies (\ref{row}) and which is such that $%
x\left( k\right) >0$ for $k\in \left[ 1,T\right] $ with $x\left( 0\right)
=x\left( T+1\right) =0$. We have the following result concerning positive
solutions.

\begin{corollary}
Assume either \textbf{A1}, \textbf{A2}, \textbf{A3, A5} or \textbf{A1}, 
\textbf{A3}, \textbf{A5, A6.} For any fixed $u\in L_{M}$ there exists at
least one solution $x\in V_{u},$ $x\left( k\right) >0$ for $k\in \left[ 1,T%
\right] ,$ to problem (\ref{row})-(\ref{bd}) such that $J_{u}\left( x\right)
=\inf_{y\in E}J_{u}\left( y\right) $. Let $\left\{ u_{n}\right\}
_{n=1}^{\infty }\subset L_{M}$ be a convergent sequence of parameters, where 
$\lim_{n\rightarrow \infty }u_{n}=\overline{u}\in L_{M}$. For any sequence $%
\left\{ x_{n}\right\} _{n=1}^{\infty }$ of positive solutions $x_{n}\in
V_{u_{n}}$ to the problem (\ref{row})-(\ref{bd}) corresponding to $u_{n}$,
there exist a subsequence $\left\{ x_{n_{i}}\right\} _{i=1}^{\infty }\subset
E$ and an element $\overline{x}\in E$ such that $\lim_{i\rightarrow \infty
}x_{n_{i}}=\overline{x}$ and $J_{\overline{u}}\left( \overline{x}\right)
=\inf_{y\in E}J_{\overline{u}}\left( y\right) $. Moreover, $\overline{x}>0$
and $\overline{x}$ satisfies (\ref{row})-(\ref{bd}) with $\overline{u}$.
\end{corollary}

\begin{proof}
Since in both cases solutions exist, we need to prove only that the
solutions to (\ref{row})-(\ref{bd}) under either \textbf{A1}, \textbf{A2}, 
\textbf{A3, A5} or \textbf{A1}, \textbf{A3}, \textbf{A5, A6} are positive.
We rewrite (\ref{row}) as follows 
\begin{equation*}
-\Delta \left( p\left( k\right) \Delta x\left( k-1\right) \right) =f\left(
k,x\left( k\right) ,u\left( k\right) \right) -g\left( k\right)
\end{equation*}
and observe that by \textbf{A5} we have $-\Delta \left( p\left( k\right)
\Delta x\left( k-1\right) \right) \geq 0$. Thus the strong comparison
principle, Lemma 2.3 from \cite{agrawal}, shows that either $x\left(
k\right) \geq 0$ for $k\in \left[ 1,T\right] $ or $x\left( k\right) =0$ for $%
k\in \left[ 1,T\right] $. Since $f\left( k_{1},x\left( k\right) ,u\left(
k\right) \right) -g\left( k_{1}\right) \neq 0$ for certain $k_{1}$ we cannot
have $x=0$. Thus, we see that $x\left( k\right) >0$ for $k\in \left[ 1,T%
\right] $.
\end{proof}

We note that neither in \cite{Li} nor in \cite{galewskiCOERCIVE} positive
solutions are considered. However, in \cite{rach} by the lower-upper
function method the authors obtain the existence of positive solutions for
system (\ref{row})-(\ref{bd}) without a parameter with \textbf{A1}, \textbf{%
A3} and with assumptions that $g\left( k\right) <0$, $f\left( t,0\right)
\geq 0$ for $t\in \left[ 1,T\right] $ (replacing \textbf{A5}) and that there
exists $\alpha >0$ such that $f\left( k,y\right) \leq 0$ for all $y\geq
\alpha $ and $k=1,...,T$ (replacing A\textbf{2}).

\section{Applications for the discrete Emden-Fowler equation\label{EF}}

Now we turn to sketching some further possible applications of our results.
As an example we shall consider the discrete version of the Emden-Fowler
equation investigated with the aid of critical point theory in \cite{HeWe}.
Following the authors of \cite{HeWe} we consider\ (in $R^{T}$ with classical
Euclidean norm) the discrete equation 
\begin{equation}
\Delta \left( p\left( k-1\right) \Delta x\left( k-1\right) \right) +q\left(
k\right) x\left( k\right) +f\left( k,x\left( k\right) ,u\left( k\right)
\right) =g\left( k\right)  \label{emden}
\end{equation}
subject to a parameter $u\in L_{M}$ and with boundary conditions 
\begin{equation}
x\left( 0\right) =x\left( T\right) \text{, }p\left( 0\right) \Delta x\left(
0\right) =p\left( T\right) \Delta x\left( T\right) \text{.}  \label{emden-bd}
\end{equation}
It is assumed that\bigskip

\textbf{A7} $f\in C\left( \left[ 1,T\right] \times R\times \left[ -M,M\right]
,R\right) $, $p\in C\left( \left[ 1,T+1\right] ,R\right) ,$ $q,g\in C\left( %
\left[ 1,T\right] ,R\right) $; $g\left( k_{1}\right) \neq 0$ for certain $%
k_{1}\in \left[ 1,T\right] ;$\bigskip

\textbf{A8} there exists a constant $r\in \left( 1,2\right) $ such that 
\begin{equation}
\lim_{\left\vert y\right\vert \rightarrow \infty }\sup \frac{f\left(
k,y,u\right) }{\left\vert y\right\vert ^{r-1}}\leq 0  \label{inf}
\end{equation}%
uniformly for $u\in \left[ -M,M\right] $, $k\in \left[ 1,T\right] .$

\bigskip

Basing on ideas developed in the proof of Theorem \ref{dep_param_theo} we
formulate and prove the main result of this section. Let us denote 
\begin{equation*}
M=\left[ 
\begin{array}{cccccc}
p\left( 0\right) +p\left( 1\right) & -p\left( 1\right) & 0 & \ldots & 0 & 
-p\left( 0\right) \\ 
-p\left( 1\right) & p\left( 1\right) +p\left( 2\right) & -p\left( 2\right) & 
\ldots & 0 & 0 \\ 
0 & -p\left( 2\right) & p\left( 2\right) +p\left( 3\right) & \ldots & 0 & 0
\\ 
\vdots & \vdots & \vdots & \ddots & \vdots & \vdots \\ 
0 & 0 & 0 & \ldots & p\left( T-2\right) +p\left( T-1\right) & -p\left(
T-1\right) \\ 
-p\left( 0\right) & 0 & 0 & \ldots & -p\left( T-1\right) & p\left(
T-1\right) +p\left( 0\right)%
\end{array}
\right]
\end{equation*}
and 
\begin{equation*}
Q=\left[ 
\begin{array}{cccccc}
-q\left( 1\right) & 0 & 0 & \ldots & 0 & 0 \\ 
0 & -q\left( 2\right) & 0 & \ldots & 0 & 0 \\ 
0 & 0 & -q\left( 3\right) & \ldots & 0 & 0 \\ 
\vdots & \vdots & \vdots & \ddots & \vdots & \vdots \\ 
0 & 0 & 0 & \ldots & -q\left( T-1\right) & 0 \\ 
0 & 0 & 0 & \ldots & 0 & -q\left( T\right)%
\end{array}
\right]
\end{equation*}
For a fixed $u\in L_{M}$ we introduce the action functional for (\ref{emden}%
)-(\ref{emden-bd}) 
\begin{equation*}
J_{u}\left( x\right) =\frac{1}{2}\left\langle \left( M+Q\right)
x,x\right\rangle -\sum_{k=1}^{T}F\left( k,x\left( k\right) ,u\left( k\right)
\right) +\sum_{k=1}^{T}g\left( k\right) x\left( k\right) .
\end{equation*}
Next, we introduce the set of critical points of (\ref{emden})-(\ref%
{emden-bd}) 
\begin{equation*}
V_{u}=\left\{ x\in R^{T}:J_{u}\left( x\right) =\inf_{v\in R^{T}}J_{u}\left(
v\right) ,\text{ }\frac{d}{dx}J_{u}\left( x\right) =0\right\} .
\end{equation*}

\begin{theorem}
Assume \textbf{A7, A8} and that $M+Q$ is a positive definite matrix. For any
fixed $u\in L_{M}$ there exists at least one non trivial solution $x\in
V_{u} $ to problem (\ref{emden})-(\ref{emden-bd}). Let $\left\{
u_{n}\right\} _{n=1}^{\infty }\subset L_{M}$ be a convergent sequence of
parameters, where $\lim_{n\rightarrow \infty }u_{n}=\overline{u}\in L_{M}$.
For any sequence $\left\{ x_{n}\right\} _{n=1}^{\infty }$ of solutions $%
x_{n}\in V_{u_{n}}$ to the problem (\ref{emden})-(\ref{emden-bd})
corresponding to $u_{n}$, there exist a subsequence $\left\{
x_{n_{i}}\right\} _{i=1}^{\infty }\subset R^{T}$ and an element $\overline{x}%
\in R^{T}$ such that $\lim_{i\rightarrow \infty }x_{n_{i}}=\overline{x}$ and 
$\overline{x}\in V_{\overline{u}}$, i.e. $\overline{x}$ satisfies (\ref%
{emden})-(\ref{emden-bd}) with $\overline{u}$, 
\begin{equation*}
\Delta \left( p\left( k-1\right) \Delta \overline{x}\left( k-1\right)
\right) +q\left( k\right) \overline{x}\left( k\right) +f\left( k,\overline{x}%
\left( k\right) ,\overline{u}\left( k\right) \right) =g\left( k\right) ,
\end{equation*}%
\begin{equation*}
\overline{x}\left( 0\right) =\overline{x}\left( T\right) \text{, }p\left(
0\right) \Delta \overline{x}\left( 0\right) =p\left( T\right) \Delta 
\overline{x}\left( T\right) \text{.}
\end{equation*}
\end{theorem}

\begin{proof}
First we must show that for any fixed $u\in L_{M}$ there exists a solution
to (\ref{emden})-(\ref{emden-bd}) and next we need to show that set $V_{u}$
is bounded uniformly in $u\in L_{M}$.

Let us fix $u\in L_{M}$. By Theorem 3.4 from \cite{HeWe} applied to our
functional we get the existence of at least one solution to (\ref{emden})-(%
\ref{emden-bd}). Indeed, we recall some arguments used in \cite{HeWe} for
convenience. Let us fix $u\in L_{M}$. Fix $\varepsilon >0$. By (\ref{inf}),
we see that there exists $B>0$ such that $\frac{f\left( k,y,u\right) }{%
\left| y\right| ^{r-1}}\leq \varepsilon $ for all $k\in \left[ 1,T\right] $
and for $\left| y\right| \geq B,$ $\left| u\right| \leq M$. Then it follows
that $F\left( k,y,u\right) \leq \frac{\varepsilon }{r}\left| y\right| ^{r}$
for all $k\in \left[ 1,T\right] $ and for $\left| y\right| \geq B,$ $\left|
u\right| \leq M$. Denoting $A=\sup_{\left( k,x,u\right) \in \left[ 1,T\right]
\times \left[ -B,B\right] \times \left[ -M,M\right] }\left| f\left(
k,x,u\right) \right| $ we see that for $\left( k,y,u\right) \in \left[ 1,T%
\right] \times R\times \left[ -M,M\right] $ by the definition of $F$ 
\begin{equation*}
F\left( k,y,u\right) \leq AB+\frac{\varepsilon }{r}\left| y\right| ^{r}.
\end{equation*}
Since $M+Q$ is positive definite there exists a number $a_{M+Q}>0$ such that
for all $y\in R^{T}$ 
\begin{equation*}
\left\langle \left( M+Q\right) y,y\right\rangle \geq a_{M+Q}\left| y\right|
^{2}
\end{equation*}
Therefore, we have by Schwartz inequality for any $y\in R$%
\begin{equation}
J_{u}\left( y\right) \geq \frac{1}{2}a_{M+Q}\left| y\right| ^{2}-T\left( AB+%
\frac{\varepsilon }{r}\left| y\right| ^{r}\right) -\left| y\right| \sqrt{%
\sum_{k=1}^{T}g^{2}\left( y\right) }.  \label{unifJu}
\end{equation}
Since $r<2$, we see that $J_{u}$ is coercive. Hence it has an argument of a
minimum $x$ which satisfies (\ref{emden})-(\ref{emden-bd}). We note that $%
x\neq 0$. Indeed, if $x=0,$ then $g\left( k_{1}\right) =0$, which is a
contradiction with \textbf{A12}.

Now we see that by inequality (\ref{unifJu}) we again have for the solution $%
x_{u}$ to (\ref{emden})-(\ref{emden-bd}) 
\begin{equation*}
\frac{1}{2}a_{M+Q}\left\vert x_{u}\right\vert ^{2}-T\frac{\varepsilon }{r}%
\left\vert x_{u}\right\vert ^{r}-\left\vert x_{u}\right\vert \sqrt{%
\sum_{k=1}^{T}g^{2}\left( y\right) }\leq J_{u}\left( x_{u}\right) \leq ABT.
\end{equation*}%
Thus the reasoning from the second part of the proof of Theorem \ref%
{dep_param_theo} now applies.
\end{proof}

We conclude the paper with some examples and remarks concerning the results
obtained in this work.

\begin{example}
Let $l$ be any natural number and let $q,r\in C\left( R,R_{+}\right) $ be
bounded. Function $f\left( k,x,u\right) =q\left( k\right) h\left( x\right)
r\left( u\right) $ with 
\begin{equation*}
h\left( x\right) =\left\{ 
\begin{array}{cc}
x^{2l}, & x\leq 0 \\ 
-x^{2l}, & x>0%
\end{array}
\right.
\end{equation*}
does not satisfy \textbf{A2}, but it satisfies \textbf{A4}.
\end{example}

\begin{example}
Let $q,r\in C\left( R,R_{+}\right) $, where $r$ is a bounded function.
Function $f\left( k,x,u\right) =q\left( k\right) h\left( x\right) r\left(
u\right) $ with 
\begin{equation*}
h\left( x\right) =\left\{ 
\begin{array}{cc}
-\frac{x+1}{1+x^{4}}, & x<0 \\ 
-1, & x\geq 0%
\end{array}%
\right.
\end{equation*}%
satisfies \textbf{A6}. Indeed, in this case 
\begin{equation*}
H\left( x\right) =\left\{ 
\begin{array}{cc}
\frac{1}{8}\sqrt{2}\ln \frac{x^{2}+x\sqrt{2}+1}{x^{2}-x\sqrt{2}+1}+\frac{1}{4%
}\sqrt{2}\arctan \left( x\sqrt{2}+1\right) +\frac{1}{4}\allowbreak \sqrt{2}%
\arctan \left( x\sqrt{2}-1\right) , & x<0 \\ 
-x, & x\geq 0%
\end{array}%
\right. .
\end{equation*}%
Hence \textbf{A6}\ can be directly verified. Taking $g\in C\left( R,\left(
-\infty ,-1\right) \right) $ we see that \textbf{A5} is also satisfied.
\end{example}

\end{document}